\documentclass[a4paper,12pt,onecolumn]{article}

\usepackage[vmargin=2cm,hmargin=2cm,headheight=14.5pt,top=2cm,headsep=.5cm]{geometry}
\usepackage{mathptmx}
\usepackage{bm}
\usepackage{url}
\usepackage{empheq}
\usepackage{stackrel}
\usepackage{wrapfig}
\usepackage{cases}
\usepackage[section]{placeins}
\usepackage{amsthm,amsmath,amscd}
\usepackage{amssymb,amsfonts}
\usepackage{makeidx}
\usepackage[all]{xy}
\usepackage{mathptmx}
\usepackage{perpage}
\usepackage[symbol]{footmisc}

\MakePerPage[2]{footnote}
\usepackage{graphicx}
\DeclareMathSizes{12}{12}{8}{6}

\usepackage{cite}

\newtheoremstyle{ptheorem}{1em}{0em}{\itshape}{}{\bfseries}{.}{.5em}{}
\theoremstyle{ptheorem}
\newtheorem{thm}{Theorem}[section]

\newtheorem{lem}[thm]{Lemma}
\newtheorem{cor}[thm]{Corollary}

\theoremstyle{definition}
\newtheorem{dfn}{Definition}[section]

\theoremstyle{remark}
\newtheorem{exa}{Example}[section]

\newtheorem{rem}{Remark}[section]

\DeclareMathOperator{\Id}{Id}

\DeclareMathOperator{\dif}{d}

\newcommand{\cC}{{\mathcal C}}

\newcommand{\bN}{{\mathbb N}}

\newcommand{\bQ}{{\mathbb Q}}
\newcommand{\bR}{{\mathbb R}}

\newcommand{\bZ}{{\mathbb Z}}

\renewcommand{\a}{\alpha}
\renewcommand{\b}{\beta}
\renewcommand{\c}{\gamma}
\renewcommand{\l}{\lambda}
\newcommand{\e}{\epsilon}

\let\foo\phi
\let\phi\varphi
\let\varphi\foo

\newcommand{\ol}{\overline}

\renewcommand{\d}{\delta}

\renewcommand{\(}{\left(}
\renewcommand{\)}{\right)}
\renewcommand{\[}{\left[}
\renewcommand{\]}{\right]}

\newcommand{\til}{\tilde}

\begin{document}
\title{Periodic solutions for some phi-Laplacian and reflection equations\footnote{Partially supported by FEDER and Ministerio de Educaci\'on y Ciencia, Spain, project MTM2010-15314}}

\author{
Alberto Cabada \, and F. Adri\'an F. Tojo\footnote{Supported by  FPU scholarship, Ministerio de Educaci\'on, Cultura y Deporte, Spain.} \\
\normalsize
Departamento de An\'alise Ma\-te\-m\'a\-ti\-ca, Facultade de Matem\'aticas,\\ 
\normalsize Universidade de Santiago de Com\-pos\-te\-la, Spain.\\ 
\normalsize e-mail: alberto.cabada@usc.es, fernandoadrian.fernandez@usc.es}
\date{}

\maketitle

\begin{abstract}
This work is devoted to the study of the existence and periodicity of solutions of initial differential problems, paying special attention to the explicit computation of the period. These problems are also connected with some particular initial and boundary value problems with reflection, which allows us to prove existence of solutions of the latter using the existence of the first.
\end{abstract}

\section{Introduction}

The idea behind this paper appeared in another work of the authors \cite{Cab4} where the following lemmas were proved.
\begin{dfn}
If $A\subset\bR$, a function $\phi:A\to A$ such that $\phi\ne\Id$ and $\phi\circ\phi=\Id$ is called an \textbf{involution}.
\end{dfn}
Let us consider the problems
\begin{equation}\label{eqinv1}
x'(t)=f(x(\phi(t))), \quad x(c)=x_c
\end{equation}
and
\begin{equation}\label{ode1}
x''(t)=f'(f^{-1}(x'(t)))f(x(t))\phi'(t), \quad x(c)=x_c, \;x'(c)=f(x_c).
\end{equation}
\begin{lem}[{\cite[Lemma 2.1]{Cab4}}]\label{lem1b}
Let $(a,b)\subset\bR$ and let $f:\bR\to\bR$ be a diffeomorphism. Let $\phi\in\cC^1((a,b))$ be an involution. Let $c$ be a fixed point of $\phi$. Then $x$ is a solution of the first order differential equation with involution
(\ref{eqinv1}) if and only if $x$ is a solution of the second order ordinary differential equation
(\ref{ode1}).\end{lem}
Furthermore, a version of Lemma \ref{lem1b} can be proved for the case with periodic boundary value conditions.\par
Let us consider the equations
\begin{equation}\label{eqinvb}
x'(t)=f(x(\phi(t))), \quad x(a)=x(b)
\end{equation}
and
\begin{equation}\label{odeb}
x''(t)=f'(f^{-1}(x'(t)))f(x(t))\phi'(t), \quad x(a)=x(b)=f^{-1}(x'(a)).
\end{equation}
\begin{lem}[{\cite[Lemma 2.2]{Cab4}}]\label{lem2}
Let $[a,b]\subset\bR$ and let $f:\bR\to\bR$ be a diffeomorphism. Let $\phi\in\cC^1([a,b])$ be an involution such that $\phi([a,b])=[a,b]$. Then $x$ is a solution of the first order differential equation with involution
(\ref{eqinvb}) if and only if $x$ is a solution of the second order ordinary differential equation
(\ref{odeb}).\end{lem}
\begin{rem}\label{remlemex}Although not stated in \cite{Cab4}, it is important to notice that the proofs of Lemmas \ref{lem1b} and $\ref{lem2}$ are still valid if we weaken the regularity hypothesis on $f$ and $f^{-1}$ to $f$ and $f^{-1}$ absolutely continuous and $f$ locally Lipschitz.
\end{rem}
Linear problems with involutions, similar to problem \eqref{eqinv1}, have also been studied in \cite{Cab4,CabToj,CabToj2}.
Observe that, from problem \eqref{odeb} we have that
$$0=\frac{x''(t)}{f'(f^{-1}(x'(t)))}- \,f(x(t))\phi'(t)=(f^{-1})'(x'(t))x''(t)- \,f(x(t))\phi'(t)=(f^{-1}\circ x')'(t)- \,f(x(t))\phi'(t).$$
So, clearly, problem \eqref{odeb} is equivalent to the problem
\begin{equation}\label{eqperbis1} (f^{-1}\circ x')'(t)-\phi'(t)f(x(t))=0,\quad x(a)=x(b),\quad x'(a)=f(x(a)). \end{equation}
Which involves the $f^{-1}$-Laplacian $(f^{-1}\circ x')'$, although, contrary to most literature, the other term in the equation does not involve $f^{-1}$ but $f$. As we will see, this is not more than a further generalization in the line of the $p-q$-Laplacian.\par
Problems concerning the $\varphi$-Laplacian (or, particularly, the $p$-Laplacian) have been studied extensively in recent literature. Dr\'abek, Man\'asevich and others study the eigenvalues of problems with the $p$-Laplacian in \cite{Dra4,Bin,Dra,Dra2, Pin} using variational methods. The existence of positive solutions is treated in \cite{Dra3}, the existence of an exact number of solutions in \cite{San} and topological existence results can be found in \cite{Pin2}. Anti-maximum principles and sign properties of the solutions are studied in \cite{CabCid, CabLom}.  In \cite{Chv} they study a variant of the $p$-Laplacian equation with an approach based on variational methods, in \cite{BogDos} they study the eigenvalues of the Dirichlet problem and in \cite{DosOz} they find some oscillation criteria for equations with the $p$-Laplacian.
\par
The $\varphi$-Laplacian is studied from different points of view in several papers, e.\,g. \cite{Hei,Ore,Kar,CabSta,Aga,Cec,Liu}.
Actually, if we consider the problem with the $f^{-1}$-Laplacian
\begin{equation}\label{eqper3} (f^{-1}\circ x_{c}')'(t)+f(x_{c}(t))=0,\quad \quad x_{c}(a)={c},\quad x_{c}'(a)=f({c}),\end{equation}
and we assume there exist $\ol {c}_1,\ol {c}_2\in\bR$, $\ol {c}_1<\ol {c}_2$, such that a unique solution of problem \eqref{eqper3} exists for every ${c}\in[\ol {c}_1,\ol {c}_2]$ and $(x_{\ol {c}_1}(b)-\ol {c}_1)(x_{\ol {c}_2}(b)-\ol {c}_2)<0$, then problem \eqref{odeb} must have at least a solution due to the continuity of $x_{c}$ on ${c}$ and Bolzano's theorem. For this reason we will be interested in studying the properties of problem \eqref{eqper3} and its solutions in this paper. In the sections to come we study this problem and more general versions of it.\par

In the following section we will study the existence, uniqueness and periodicity of solutions of problem \eqref{eqperg} and in Section 3 we will apply these results to the case of problems with reflection.

\section{General solutions}

First, we write in a general way the solutions of equations involving the $g-f$-Laplacian.

Let $\tau_i,\sigma_i\in[-\infty,\infty]$, $i=1,\dots,4$, $\tau_1<\tau_2$, $\sigma_1<\sigma_2$, $\tau_3<\tau_4$, $\sigma_3<\sigma_4$. Let $f:(\tau_1,\tau_2)\to(\sigma_1,\sigma_2)$ and $g:(\tau_3,\tau_4)\to(\sigma_3,\sigma_4)$ be invertible functions such that $f$ and $g^{-1}$ are continuous. Assume there is $s_0\in(\tau_1,\tau_2)$ such that $f(s_0)=0$ and define $F(t):=\int_{s_0}^tf(s)\dif s$. Observe that $F$ is $0$ at $s_0$ and of constant sign everywhere else. The following Lemma is an straightforward application of the properties of the integral.\par
The following result holds immediately from the properties of continuous real functions.
\begin{lem} If $f$ is continuous, invertible and increasing (decreasing) then
$F_-\equiv F|_{(-\infty,s_0]}$  is strictly decreasing (increasing) and $F_+\equiv F|_{[s_0,+\infty)}$ is strictly increasing (decreasing). Furthermore, if $\tau_1=-\infty$, $F(-\infty)=+ \infty$ $(-\infty)$ and if $\tau_2=+\infty$, $F(+\infty)=+ \infty$ $(-\infty)$.
\end{lem}
All the same, define $G(t):=\int_{g^{-1}(\{0\})}^tg^{-1}(s)\dif s$ and consider the problem
\begin{equation}\label{eqperg} (g\circ x')'(t)+\,f(x(t))=0,\quad \text{a.\,e. }t\in\bR, \quad x(a)=c_1,\quad x'(a)=c_2, \end{equation}
for some fixed $c_1,c_2\in\bR$.\par
A \textit{solution} $x$ of problem \eqref{eqperg} will be $x\in \cC^{1}(I)$ such that is absolutely continuous on $I$ where  $I$ is an open interval with $a\in I$. The solution must further satisfy that the equation in problem \eqref{eqperg} satisfied a.\,e and the initial conditions are satisfied as well.
\begin{thm}\label{thmex} Let $f:(\tau_1,\tau_2)\to(\sigma_1,\sigma_2)$ and $g:(\tau_3,\tau_4)\to(\sigma_3,\sigma_4)$ be invertible functions such that $f$ and $g^{-1}$ are continuous and assume $0\in(\tau_1,\tau_2)\cap(\tau_3,\tau_4)$, $f(0)=0$, $g(0)=0$, $f$ and $g$ increasing, $F(c_1)+G(g(c_2))<\min\{G(\sigma_3),G(\sigma_4)\}$. Then there exists a unique local solution of problem \eqref{eqperg}.\par
Furthermore, if $F(c_1)+G(g(c_2))<\min\{F(\tau_1),F(\tau_2)\}$,
then such solution is defined on the whole real line and is periodic of smallest period
\begin{equation}\label{pfor}T:=\int_{F^{-1}_-(G(g(c_2))+F(c_1))}^{F_+^{-1}\(G(g(c_2))+F(c_1)\)}\[\frac{1}{g^{-1}\circ G_+^{-1}(G(g(c_2))+F(c_1)-F(r))}-\frac{1}{g^{-1}\circ G_-^{-1}(G(g(c_2))+F(c_1)-F(r))}\]\dif r.\end{equation}
\end{thm}
\begin{proof} For the first part of the Theorem and without loss of generality, we will prove the existence of solution in an interval of the kind $[a,a+\d)$, $\d\in\bR^+$. The proof would be analogous for an interval of the kind $(a-\d, a]$.\par
Let $y(t)=g(x'(t))$. Then problem \eqref{eqperg} is equivalent to
$$x'(t)=g^{-1}(y(t)),\quad y'(t)=-\,f(x(t)),\quad x(a)=c_1,\ y(a)=g(c_2).$$
Hence,
$$f(x(t))x'(t)+g^{-1}(y(t))y'(t)=0,$$
so, integrating both sides from $a$ to $t$,
\begin{equation*}\label{poteq}F(x(t))+G(y(t))=k,\end{equation*}
where $k=F(c_1)+G(g(c_2))$. That is, undoing the change of variables,
 \begin{equation}\label{eqprior} G(g(x'(t)))=G(g(c_2))+F(c_1)-F(x(t)).\end{equation}
 
 \par If $c_1=c_2=0$ it is clear that the only possible solution is $x\equiv 0$. Assume, without loss of generality, that $c_2$ is non-negative and $c_1$ negative (the other cases are similar). If $c_2=0$ then, integrating \eqref{eqperg},
\begin{equation*}\label{eqint}g\circ x'(t)=-\int_a^tf(x(s))\dif s,\end{equation*}
which implies $x'$ is positive in some interval $[a,a+\d)$. If $c_2$ is positive, then $x'$ has to be positive at least in some neighborhood of $a$, so, in a right neighborhood of $a$, we can solve for $g\circ x'$ in \eqref{eqprior} as
\begin{equation}\label{eqgxp}g\circ x'(t)=G_+^{-1}(F(c_1)-F(x(t))+G(g(c_2))).\end{equation}
In order to solve for $x'$ in \eqref{eqgxp}, we need $F(c_1)+G(g(c_2))<G(\sigma_4)$. Then,
\begin{equation}\label{eqx'} x'(t)=g^{-1}\circ G_+^{-1}(F(c_1)-F(x(t))+G(g(c_2))).\end{equation}
Integrating between $a$ and $t$,
$$t=\int_a^t\frac{x'(s)}{g^{-1}\circ G_+^{-1}(F(c_1)-F(x(s))+G(g(c_2)))}\dif s+a=H_+(x(t)),$$
where
$$H_+(r):=\int_{c_1}^r\frac{1}{g^{-1}\circ G_+^{-1}(F(c_1)-F(s)+G(g(c_2)))}\dif s+a.$$
$H_+$ is strictly increasing in its domain due to the positivity of the denominator in the integrand. Hence, for $t$ sufficiently close to $a$,
$$x(t)=H_+^{-1}(t).$$
Therefore, a solution of problem \eqref{eqperg} exists and is unique on an interval $[a,a+\delta)$.\par
If we assume $F(c_1)+G(g(c_2))<\min\{F(\tau_1),F(\tau_2)\}$, $c_2>0$ (the case $c_2=0$ is similar), $H_+$ is well defined on $$I:=\(F_-^{-1}(F(c_1)+G(g(c_2))),F_+^{-1}(F(c_1)+G(g(c_2)))\).$$
Now, we study the range of $H_+$. $g(x'(t))$ is positive as long as $x'(t)$  is negative. Hence, consider $$t_0:=\sup\{t\in[a,+\infty)\ :\ x'(s)>0 \text{ for a.\,e. } s\in[a,t)\}\in[a,+\infty].$$
$G$ is positive on non-zero values, so equation \eqref{eqprior} implies that, $F(x(t))< G(g(c_2))+F(c_1)$ for all $t\in(a,t_0)$.\par
Assume $t_0=+\infty$. Now, $x'(t)>0$ a.\,e. in $[a,+\infty)$ so there exists $x(+\infty)\in(c_1,F^{-1}_+(G(g(c_2))+F(c_1))]$.\par
On the other hand, since $x$ is increasing in $[a,+\infty)$ and $c_1<0$, by equation \eqref{eqx'} we have that $x'$ is increasing as long as $x$ is positive. This means that, eventually (in finite time), $x$ will be positive and therefore, $x'$ is decreasing in $[\til a,+\infty)$ for $\til a$ big enough, so there exists $x'(+\infty)\ge 0$. If we assume $x'(+\infty)=\d>0$, this implies that $x(+\infty)=+\infty$, for there would exist $M\in\bR$ such that $x'(t)>\d/2$ for every $t\ge M$, so $x'(+\infty)=0$. Taking the limit $t\to+\infty$ in equation \eqref{eqprior}, $x(+\infty)=F^{-1}_+(G(g(c_2))+F(c_1))$.\par
Now, take $\e\in(0,f(x(+\infty)))$. Since $g\circ x'(+\infty)=0$ and $g\circ x'$ is continuous, there exists $M\in\bR^+$ such that $|g(x'(M_2))-g(x'(M_1))|<\e$ for every $M_1,$ $M_2>M$. Since $f$ is continuous, there exits $\til M>M$ such that $f(x(M_3))>\e$ for every $M_3>\til M$. Take $M_3$ in such a way. Then, integrating equation \eqref{eqperg} between $M_3$ and $M_3+1$,
$$(g\circ x')(M_3+1)-(g\circ x')(M_3)=\int_{M_3}^{M_3+1}f(x(s))\dif s>\e,$$
a contradiction. Therefore, $t_0\in\bR$.\par
Observe that $x'(t_0)=0$, so $x$ attains its maximum at $t_0$ and $x(t_0)=F^{-1}_+(G(g(c_2))+F(c_1))$ by equation \eqref{eqprior}, that is, $x(t_0)=\sup I$. In order for this value to be well defined it is necessary that $G(g(c_2))+F(c_1)\le F(\tau_2)$.\par
Now, we have that $H_+$ is well defined at $\sup I$ (assuming it is defined continuous at that point). Indeed,
$$t_0=\lim_{t\to t_0}H_+(x(t))=H_+(F^{-1}_+(G(g(c_2))+F(c_1))).$$
 We prove now that there there is a neighborhood $(t_0,t_0+\e)$ where $x'$ is negative, which means that we can take
 $$t_1:=\sup\{t\in[t_0,+\infty)\ :\ x'(s)<0 \text{ for a.\,e. } s\in[t_0,t)\}.$$
Fix $\d<f(x(t_0))$ and take $\e$ such that $f(x(t))>\d$ in $(t_0, t_0+\e)$. Take $t\in(t_0, t_0+\e)$, then, integrating equation \eqref{eqperg} between $t_0$ and $t$,
$$g(x'(t))=-\int_{t_0}^tf(x(s))\dif s<-\e(t-t_0)<0.$$
We deduce that $t_1<+\infty$ by the same kind of reasoning we used to prove $t_0<+\infty$. Observe that $x'(t_1)=0$ and $x(t_1)=F_-^{-1}\(G(g(c_2))+F(c_1)\)$. This last equality comes from evaluating equation \eqref{eqprior} at $t_1$ and Rolle's Theorem as we show now: the other possibility would be $x(t_1)=F_+^{-1}\(G(g(c_2))+F(c_1)\)$. Observe that, by equation \eqref{eqx'}, $x'$ is continuous, so $x\in\cC^1$. Since $x(t_0)=x(t_1)$, there would exist $\til t\in(t_0,t_1)$ such that $x'(\til t)=0$, a contradiction.\par
Now, we have that  $x'(t)=g^{-1}\circ G_-^{-1}(G(g(c_2))+F(c_1)-F(x(t)))$, that is, $$1=x'(t)/g^{-1}\circ G_-^{-1}(G(g(c_2))+F(c_1)-F(x(t))).$$ Thus,
$$t_1-t_0=\int_{t_0}^{t_1}\frac{x'(s)\dif s}{g^{-1}\circ G_-^{-1}(G(g(c_2))+F(c_1)-F(x(s)))}=\int_{F^{-1}_+(G(g(c_2))+F(c_1))}^{F_-^{-1}\(G(g(c_2))+F(c_1)\)}\frac{\dif r}{g^{-1}\circ G_-^{-1}(G(g(c_2))+F(c_1)-F(r))}.$$
If we define
$$H_{-}(s):=\int_{F^{-1}_+(G(g(c_2))+F(c_1))}^{s}\frac{\dif r}{g^{-1}\circ G_-^{-1}(G(g(c_2))+F(c_1)-F(r))}+t_0,$$
$H_{-}$ is strictly decreasing in its domain and $x(t)=H_{-}^{-1}(t)$ for $t\in[t_0,t_1]$.\par
We can again deduce that
$$t_2:=\sup\{t\in[t_1,+\infty)\ :\ x'(s)>0 \text{ for a.\,e. } s\in[t_1,t)\}<+\infty.$$
Using the positivity and growth conditions of the functions involved, it is easy to check that $x(t_1)=F_-^{-1}\(G(g(c_2))+F(c_1)\)<c_1<F_+^{-1}\(G(g(c_2))+F(c_1)\)=x(t_2)$, so there exists a unique $b\in(t_1,t_2)$ such that $x(b)=c_1$.
Now,
$$b-t_1=\int_{t_1}^{b}\frac{x'(s)\dif s}{g^{-1}\circ G_+^{-1}(G(g(c_2))+F(c_1)-F(x(s)))}=\int_{F_-^{-1}\(G(g(c_2))+F(c_1)\)}^{c_1}\frac{\dif r}{g^{-1}\circ G_+^{-1}(G(g(c_2))+F(c_1)-F(r))}.$$
Defining $T:=b-a$ and extending $x$ periodically  in the following way (we have $x$ already defined in $[a,a+T]$),
$$x(t)=x\(t-\left\lfloor\frac{t-a}{T}\right\rfloor T\),$$
where $\lfloor t\rfloor:=\sup\{k\in\bZ\ :\ k\le t\}$, it is easy to check that $x$, extended in such a way,  is a global periodic solution of problem \eqref{eqperg}. Take $z(t):=x(t-T)$, $t\in\bR$, we show that $z$ is a solution of the problem in $[a+T,a+2T]$.
\begin{align*}0 & =(g\circ x')'(t)+f(x(t))=(g\circ z')'(t+T)+f(z(t+T))\end{align*}
This is equivalent to
$$(g\circ z')'(t)+f(z(t))=0\quad\text{ for a.\,e}\ t\in\bR.$$
Also,
$$z(a)=x(a+T)=x(b)=c_1,\quad z'(a)=x'(b)=g^{-1}\circ G_+^{-1}(F(c_1)-F(x(b))+G(g(c_2)))=c_2.$$
\end{proof}
\begin{rem}\label{dgtrem} A similar argument can be done for the case $f$ and $g$ have different growth  type (e.\,g. $f$ increasing and $g$ decreasing), but taking the negative branch of the inverse function $G^{-1}$ in \eqref{eqx'}.
\end{rem}
\begin{rem}\label{disprem} In the hypothesis of theorem \ref{thmex}, if instead of $g(0)=f(0)=0$ we have that $g(s_0)=f(s_0)=0$, define $\til f(x):=f(x+s_0)$, $\til g(x):=g(x+s_0)$. Then $\til f(0)=\til g(0)=0$ and problem \eqref{eqperg} is equivalent to
$$(\til g\circ v')'(t)+\til f(v(t))=0,\quad v(a)=c_1-s_0,\quad v(a)=c_2,$$
with $v(t)=x(t)-s_0$. Hence, we can apply Theorem \ref{thmex} to this case.
\end{rem}
\begin{rem} Using the notation of the Theorem, the explicit form of the solution of problem \eqref{eqperg} is given by
$$x(t)=\begin{dcases}H^{-1}_+\(t-\left\lfloor\frac{t-a}{T}\right\rfloor T\), & t\in [a+2Tk,a+(2k+1)T],\ k\in\bZ,\\H^{-1}_-\(t-\left\lfloor\frac{t-a}{T}\right\rfloor T\), & t\in [a+(2k-1)T,a+2kT],\ k\in\bZ,
\end{dcases}$$
\end{rem}
\begin{rem} Consider the following particular case of problem \eqref{eqperg} with  $f(0)=0$, $g(0)=0$, $f$ and $g$ increasing and the hypothesis for a unique global solution of the following problem are satisfied in Theorem \ref{thmex}.
\begin{equation}\label{eqarcsin} (g\circ x')'(t)+\,f(x(t))=0,\quad \quad x(0)=0,\quad x'(0)=1.\end{equation}
It is clear that, in the case $g(x)=f(x)=x$, the unique solution of problem \eqref{eqarcsin} is $\sin(t)$, which suggests the definition of the $\sin_{g,f}$ function as the unique solution of problem \eqref{eqarcsin} for general $g$ and $f$. Correspondingly,
$$\arcsin^+_{g,f}(r):=H_+(r)$$
This function, defined as such, coincides with the $\arcsin_p$ function defined in \cite{Kle,Bus} for the $p$-Laplacian $f(x)=g(x)=|x|^{p-2}x$, the function $\arcsin_{p,q}$ defined in \cite{Edm,Bha,Jia} for the $p-q$-Laplacian $f(x)=|x|^{q-2}x$, $g(x)=|x|^{p-2}x$, which first appeared with a slightly different definition in \cite{Dra}, and the hyperbolic version of this function, also in \cite{Bha,Jia}, which corresponds to the case $f(x)=|x|^{q-2}x$, $g(x)=-|x|^{p-2}x$. \cite{Tak} derives generalized Jacobian functions in a similar way, defining
$$\operatorname{arcsn}_{p,q}(t,k):=\int_{0}^t\frac{1}{\sqrt[p]{(1-s^q)(1-k^qs^q)}}\dif s,$$
of which the inverse (see \cite[Proposition 3.2]{Tak}) is precisely a solution of 
$$(f_p\circ x'(t))'+\frac{q}{p^*}f_q(x(t))(1+k^q-2k^q|x(t)|^q)=0.$$
where $f_r$ is the $r$-Laplacian for $r=p,q$ and $p^*p=p^*+p$. Observe this case is also covered by our definition.\par
In all of the aforementioned works they are interested on the inverse of the $\arcsin_{g,f}$ function, the $\sin_{g,f}$ function, which they extend to the whole real line by symmetry and periodicity. Observe that in our case $f$ and $g$ need not to be odd functions, contrary to the above examples, but we can still give the definition of the $\sin_{g,f}$ function in the whole real line. Also, this lack of symmetry gives rise to a richer set of right inverses of $\sin_{g,f}$, for instance,\par
$$\arcsin^-_{g,f}(r):=H_-(r),$$
In general, if we have a problem of the kind
$$\Phi((g\circ x')',x(t))=0;\quad x(0)=0,\ x'(0)=1,$$
and we know it has a unique solution in a neighborhood of $0$, then we can define $\sin_{g,\Phi}$ as such unique solution and its inverse, in a neighborhood of $0$, $\arcsin_{g,\Phi}$.
\end{rem}

We now study the periodicity of the solutions of problem \eqref{eqperg} with the functions and constants defined in the previous section.
\subsection{A particular case}
Having in mind problem \eqref{eqper3}, we now consider a particular case of problem \eqref{eqperg} for the rest of this section. Assume $f$ is invertible and both $f$ and $f^{-1}$ are continuous. For convenience, assume also that $f$ is increasing and $f(0)=0$. Consider the following problem.
\begin{equation}\label{eqpergpar} (f^{-1}\circ x')'(t)+\l\,f(x(t))=0,\quad x(a)=c, \ x'(a)=f(c),\end{equation}
where $\l\in\bR^+$.\par
The following corollary is just the restatement of Theorem \ref{thmex} for this particular case.
\begin{cor}\label{thmexc} Let $f:(\tau_1,\tau_2)\to(\sigma_1,\sigma_2)$ be an invertible function such that $f$ is continuous and assume $0\in(\tau_1,\tau_2)$, $f(0)=0$ and $f$ increasing, $\l>0$, $(1+\l)F(c)<\min\{F(\tau_1),F(\tau_2)\}$. Then there exists a unique local solution of problem \eqref{eqpergpar}.\par
Furthermore, if $(1+\l^{-1})F(c)<\min\{F(\tau_1),F(\tau_2)\}$,
then such solution is defined on $\bR$ and is periodic of first period
\begin{equation}\label{pfor2}T:=\int_{F_-^{-1}\((1+\lambda^{-1})F(c)\)}^{F_+^{-1}\((1+\lambda^{-1})F(c)\)}\[\frac{1}{f(F_+^{-1}((1+\lambda)F(c)-\lambda\,F(r)))}-\frac{1}{f(F_-^{-1}((1+\lambda)F(c)-\lambda\,F(r)))}\]\dif r.\end{equation}
\end{cor}
There are some particular cases where the formula \eqref{pfor2} can be simplified.\par
If $f$ is odd then $F$ is even and, with the change of variables $r=c\,s$, we have that expression \eqref{pfor2} becomes
$$T=\int_{0}^{\frac{F_+^{-1}((1+\lambda^{-1})F(c))}{c}}\frac{4\,c\dif r}{f(F_+^{-1}((1+\lambda)F(c)-\lambda\,F(c\,r)))}.$$\par
Also, if we further assume that $f$ satisfies is defined in $\bR$ and that $f(rt)=h(r)f(t)$ for every $r,t\in\bR$ (see Remark \ref{remhomomorphism} for a classification of such functions) and some function $h$, then
$$F(rt)=\int_0^{rt}f(s)\dif s=\int_0^{t}f(rs)r\dif s=r\,h(r)\int_0^{t}f(s)\dif s=\,rh(r)F(t),$$
so $F$ satisfies the same kind of property for $\til h(r)=r\,h(r)$. Clearly,
$$F_-^{-1}(\til h(r)t)=r\,F_-^{-1}(t), \quad F_+^{-1}(\til h(r)t)=r\,F_+^{-1}(t).$$ Observe that $\til h(r)=F(r)/F(1)$, and therefore $\til h|_{(-\infty,0]}$, $\til h|_{[0,+\infty)}$ are invertible. Also, $\til h_+^{-1}(t)=F_+^{-1}(t\,F(1))$ for $t>0$. Hence,
\begin{align*}
\frac{F_+^{-1}((1+\lambda^{-1}))F(c))}{c} & =\frac{F_+^{-1}(\til h(\til h_+^{-1}(1+\lambda^{-1})F(c))}{c}=\frac{\til h_+^{-1}(1+\lambda^{-1})F_+^{-1}(F(c))}{c} \\ &=\til h_+^{-1}(1+\lambda^{-1}) =F_+^{-1}((1+\lambda^{-1})F(1)).
\end{align*}
All the same, $F_-^{-1}((1+\lambda^{-1})F(c))/c=F_-^{-1}((1+\lambda^{-1})F(1))$.\par
Also,
\begin{align*}
& f(F_+^{-1}((1+\lambda)F(c)-\lambda\,F(c\,r)))=f(F_+^{-1}((1+\lambda)\til h(c)F(1)-\lambda\,\til h(c)F(r)))\\ = & f(F_+^{-1}(\til h(c)[(1+\lambda)F(1)-\lambda\,F(r))])=f(c\,F_+^{-1}((1+\lambda)F(1)-\lambda\,F(r))) \\ = & h(c)f(F_+^{-1}((1+\lambda)F(1)-\lambda\,F(r)))=(f(c)/f(1))f(F_+^{-1}((1+\lambda)F(1)-\lambda\,F(r))).
\end{align*}
\par With these considerations in mind, we have that we can further reduce expression \eqref{pfor2} to
$$T(c,\lambda)=\frac{4\,c\,f(1)}{f(c)}\int_{0}^{\frac{F_+^{-1}((1+\lambda^{-1})F(1))}{c}}\frac{\dif r}{f(F_+^{-1}((1+\lambda)F(1)-\lambda\,F(r)))}.$$
\begin{exa}\label{exaplap}
Let $f(t):=|t|^{p-2}t$, $p>1$. Then
$$T(c,\lambda,p)=4\,c^{2-p}\int_{0}^{(1+\lambda^{-1})^\frac{1}{p}}[1+\lambda-\lambda\,r^p]^{\frac{1-p}{p}}\dif r.$$
Observe that with the change of variable $r=(1+\lambda^{-1})^{\frac{1}{p}}s$ we have that
\begin{align*}T(c,\lambda,p) & =4\,c^{2-p}\int_{0}^1(1+\lambda^{-1})^{\frac{1}{p}}[(1+\lambda)(1-\,s^p)]^{\frac{1-p}{p}}\dif s=4\,c^{2-p}\l^{-\frac{1}{p}}(1+\l)^{\frac{2}{p}-1}\int_{0}^1(1-\,s^p)^{\frac{1-p}{p}}\dif s \\ & =4\,c^{2-p}\l^{-\frac{1}{p}}(1+\l)^{\frac{2}{p}-1}\frac{\Gamma\(\frac{1}{p}\)^2}{p\,\Gamma\(\frac{2}{p}\)} \end{align*}
$T$ is increasing on $|c|$ if $p\in(1,2)$ and decreasing on $|c|$ if $p>2$ and independent of $|c|$ if $p=2$.\par If we take $\lambda=1$,
$$T(c,1,p)=2^{\frac{2}{p}+1}\,c^{2-p}\frac{\Gamma\(\frac{1}{p}\)^2}{p\,\Gamma\(\frac{2}{p}\)}.$$
In particular, $T(c,1,2)=2\pi$ (independently of $c$).\par
We can also consider the dependence of $T$ on $\lambda$. We do this study for this particular example and in the following section we develop a general theory.
$$\frac{\partial T}{\partial \lambda}(c,\lambda,p)=\frac{4\,c^{2-p}}{p \lambda } \left(1+\frac{1}{\lambda }\right)^{-\frac{1}{p}} (1+\lambda )^{\frac{1-2p}{p}} (1+(p-1) \lambda )\int_{0}^1\left(1-s^p\right)^{\frac{1-p}{p}}\dif s>0.$$
Therefore the period $T$ is increasing on $\l$.
\end{exa}
\begin{rem}\label{remhomomorphism}
If a continuous function $f$ satisfies that $f(rt)=h(r)f(t)$, we can obtain the explicit expression of $f$. Let $c=f(1)$,  $g(t):=f(t)/f(1)$ and $\a=\ln g(e)$. Then $g(t\,s)=g(t)g(s)$. Also, for $t\ne 0$, $1=g(1)=g(t/t)=g(t)g(1/t)$ and therefore $g(t^{-1})=g(t)^{-1}$. If $n\in\bN$, $g(t^n)=g(t)^n$, so, for $t\ge0$, $g(t)=g(t^{\frac{n}{n}})=g(t^{\frac{1}{n}})^n$ and $g(t^{\frac{1}{n}})=g(t)^\frac{1}{n}$. Hence, $g(t^\frac{p}{q})=g(t)^\frac{p}{q}$ for every $p,q\in\bN$, $q\ne 0$ and, by the density of $\bQ$ in $\bR$ and the continuity of $f$, $g(t^r)=g(t)^r$ for all $t\ge0$, $r\in\bR^+$.\par
Now, for $t>0$, $g(t)=g(e^{\ln t})=g(e)^{\ln t}=e^{\ln g(e)\ln t}=t^{\ln g(e)}=t^\a$. Hence, $f(t)=\b\,t^\a$ for $t\ge0$. On the other hand, $1=g(1)=g((-1)^2)$, so $g(-1)=\pm 1$. Also, $f(-t)=g(-1)f(t)$ and thus, $f(-t)=\pm \b\,t^\a$ for $t>0$. In summary,
$$f(t)=\begin{cases} \b\,t^\a & \text{if}\ t\ge0,\\  \pm \b\,(-t)^\a & \text{if}\ t<0.\end{cases}$$
If we further ask for $f$ to be injective, $f(t)=\beta|t|^{\a-1}t$, that is, $f$ is an $\a$-laplacian.
\end{rem}
\subsection{Dependence of T on $\l$ and $c$}
Based on the approach used in Example \ref{exaplap}, we study now the dependence of $T$ on $\l$ and $c$ in a general way.\par
We continue to assume the hypotheses for \eqref{eqpergpar} and further assume that $f$ is a differentiable function. Let us divide the interval of integration in equation $\eqref{pfor}$ in $[F_-^{-1}((1+\lambda^{-1})F(c)),0]$ and $[0,F_+^{-1}((1+\lambda^{-1})F(c))]$. Observe that $F$ is injective restricted to any of the two intervals. For the nonnegative interval, taking the change of variables
$$r=F_+^{-1}\((1+\l^{-1})F(c\,s)\),$$
we have that
\begin{align*}& \int_{0}^{F_+^{-1}((1+\lambda^{-1})F(c))}\[\frac{1}{f(F_+^{-1}((1+\lambda)F(c)-\lambda\,F(r)))}-\frac{1}{f(F_-^{-1}((1+\lambda)F(c)-\lambda\,F(r)))}\]\dif r\\ = &
\int_{0}^{1}\[\frac{1}{f(F_+^{-1}((1+\lambda)[F(c)-\,F(c\,s)])}-\frac{1}{f(F_-^{-1}((1+\lambda)[F(c)-\,F(c\,s)])}\]\frac{[1+\l^{-1}]c\,f(c\,s)}{f\(F_+^{-1}\((1+\l^{-1})F(c\,s)\)\)}\dif s.
\end{align*}
All the same, with the change of variables
$$r=F_-^{-1}\((1+\l^{-1})F(c\,s)\),$$
\begin{align*}& \int_{F_-^{-1}((1+\lambda^{-1})F(c))}^{0}\[\frac{1}{f(F_+^{-1}((1+\lambda)F(c)-\lambda\,F(r)))}-\frac{1}{f(F_-^{-1}((1+\lambda)F(c)-\lambda\,F(r)))}\]\dif r\\ = &
\int_{1}^{0}\[\frac{1}{f(F_+^{-1}((1+\lambda)[F(c)-\,F(c\,s)])}-\frac{1}{f(F_-^{-1}((1+\lambda)[F(c)-\,F(c\,s)])}\]\frac{[1+\l^{-1}]c\,f(c\,s)}{f\(F_-^{-1}\((1+\l^{-1})F(c\,s)\)\)}\dif s.
\end{align*}
Now let, for $\l\in\bR^+$, $s\in[0,1]$, $c\ne0$
\begin{align*}
\a(\l,s,c)  :&=(1+\l^{-1})c\,f(c\,s),\quad\frac{\partial\a}{\partial\l}(\l,s,c)=-\l^{-2}c\,f(c\,s),\\
\b_\pm(\l,s,c)  :&=f\(F_\pm^{-1}\((1+\l^{-1})F(c\,s)\)\), \\ \frac{\partial\b_\pm}{\partial\l}(\l,s,c) & =-\l^{-2}F(c\,s)\frac{f'\(F_\pm^{-1}\((1+\l^{-1})F(c\,s)\)\)}{f\(F_\pm^{-1}\((1+\l^{-1})F(c\,s)\)\)},\\
\c_\pm(\l,s,c)  :&=f(F_\pm^{-1}((1+\lambda)[F(c)-\,F(c\,s)])), \\ \frac{\partial\c_\pm}{\partial\l}(\l,s,c) & =[F(c)-\,F(c\,s)]\frac{f'\(F_\pm^{-1}((1+\lambda)[F(c)-\,F(c\,s)])\)}{f\(F_\pm^{-1}((1+\lambda)[F(c)-\,F(c\,s)])\)}.
\end{align*}
Then
\begin{equation}\label{eqtlc}T(\l,c)=\int_0^1\a(\l,s,c)\[\frac{1}{\b_+(\l,s,c)}-\frac{1}{\b_-(\l,s,c)}\]\[\frac{1}{\c_+(\l,s,c)}-\frac{1}{\c_-(\l,s,c)}\]\dif s.\end{equation}
Therefore,
\begin{align*}\frac{\partial T}{\partial\l}(\l,c)  = & \int_0^1\left\{\frac{\partial \a}{\partial\l}(\l,s,c)\[\frac{1}{\b_+(\l,s,c)}-\frac{1}{\b_-(\l,s,c)}\]\[\frac{1}{\c_+(\l,s,c)}-\frac{1}{\c_-(\l,s,c)}\]+\right.\\ & \a(\l,s,c)\[\frac{\frac{\partial\b_-}{\partial\l}(\l,s,c)}{\b_-(\l,s,c)^2}-\frac{\frac{\partial\b_+}{\partial\l}(\l,s,c)}{\b_+(\l,s,c)^2}\]\[\frac{1}{\c_+(\l,s,c)}-\frac{1}{\c_-(\l,s,c)}\]+\\ & \left.\a(\l,s,c)\[\frac{1}{\b_+(\l,s,c)}-\frac{1}{\b_-(\l,s,c)}\]\[\frac{\frac{\partial\c_-}{\partial\l}(\l,s,c)}{\c_-(\l,s,c)^2}-\frac{\frac{\partial\c_+}{\partial\l}(\l,s,c)}{\c_+(\l,s,c)^2}\]\right\}\dif s.\end{align*}
Observe that $\a$, $f|_{[0,1]}$, $f'$, $F$, $F^{-1}_+$, $\b_+$, $\frac{\partial\b_-}{\partial\l}$, $\c_+$, $\frac{\partial\c_+}{\partial\l}$  are non-negative, while $\frac{\partial \a}{\partial\l}$, $F^{-1}_-$, $\b_-$, $\frac{\partial\b_+}{\partial\l}$, $\c_-$, $\frac{\partial\c_-}{\partial\l}$ are non-positive. In general we cannot tell the sign of $T(\l,c)$ from this expression, but making certain assumptions we can simplify it to derive information.\par
Assume now $f$ is and odd function. Then $F^{-1}_-=-F^{-1}_+$, $\b_-=-\b_+$ and $\c_-=-\c_+$, so
$$\frac{\partial T}{\partial\l}(\l,c)  = 4\int_0^1\frac{1}{\b_+(\l,s,c)\c_+(\l,s,c)}\[\frac{\partial \a}{\partial\l}(\l,s,c)- \a(\l,s,c)\(\frac{\frac{\partial\b_+}{\partial\l}(\l,s,c)}{\b_+(\l,s,c)}+\frac{\frac{\partial\c_+}{\partial\l}(\l,s,c)}{\c_+(\l,s,c)}\)\]\dif s.$$
Now, if we differentiate equation \eqref{eqtlc} with respect to $c$,
\begin{align*}\frac{\partial T}{\partial c}(\l,c)  = & \int_0^1\left\{\frac{\partial \a}{\partial c}(\l,s,c)\[\frac{1}{\b_+(\l,s,c)}-\frac{1}{\b_-(\l,s,c)}\]\[\frac{1}{\c_+(\l,s,c)}-\frac{1}{\c_-(\l,s,c)}\]+\right.\\ & \a(\l,s,c)\[\frac{\frac{\partial\b_-}{\partial c}(\l,s,c)}{\b_-(\l,s,c)^2}-\frac{\frac{\partial\b_+}{\partial c}(\l,s,c)}{\b_+(\l,s,c)^2}\]\[\frac{1}{\c_+(\l,s,c)}-\frac{1}{\c_-(\l,s,c)}\]+\\ & \left.\a(\l,s,c)\[\frac{1}{\b_+(\l,s,c)}-\frac{1}{\b_-(\l,s,c)}\]\[\frac{\frac{\partial\c_-}{\partial c}(\l,s,c)}{\c_-(\l,s,c)^2}-\frac{\frac{\partial\c_+}{\partial c}(\l,s,c)}{\c_+(\l,s,c)^2}\]\right\}\dif s.\end{align*}
Observe that
\begin{align*}
\frac{\partial\a}{\partial c}(\l,s,c) & =(1+\l^{-1})\[f(c\,s)+c\,s\,f'(c\,s)\],\\
\frac{\partial\b_\pm}{\partial c}(\l,s,c) & =(1+\l^{-1})\,s\,f(c\,s)\frac{f'\(F_\pm^{-1}\((1+\l^{-1})F(c\,s)\)\)}{f\(F_\pm^{-1}\((1+\l^{-1})F(c\,s)\)\)},\\
 \frac{\partial\c_\pm}{\partial c}(\l,s,c) & =(1+\lambda)[f(c)-s\,f(c\,s)]\frac{f'\(F_\pm^{-1}((1+\lambda)[F(c)-\,F(c\,s)])\)}{f\(F_\pm^{-1}((1+\lambda)[F(c)-\,F(c\,s)])\)}.
\end{align*}
Hence, $\frac{\partial\a}{\partial c}$, $\frac{\partial\b_+}{\partial c}$ is positive and $\frac{\partial\b_-}{\partial c}$ negative for $c\ge0$. Assume now $f$ is an odd function.
$$\frac{\partial T}{\partial c}(\l,c)  = 4\int_0^1\frac{1}{\b_+(\l,s,c)\c_+(\l,s,c)}\[\frac{\partial \a}{\partial c}(\l,s,c)- \a(\l,s,c)\(\frac{\frac{\partial\b_+}{\partial c}(\l,s,c)}{\b_+(\l,s,c)}+\frac{\frac{\partial\c_+}{\partial c}(\l,s,c)}{\c_+(\l,s,c)}\)\]\dif s.$$
\begin{exa} Let $f:(-1,1)\to\bR$, $f(x):=x/\sqrt{1-x^2},\ x\in\bR$ and consider problem \eqref{eqpergpar}\footnote{The diffeomorphism $f$ in this example has been widely studied by Bereanu and Mawhin (see, for instance, \cite{Ber}) and is known as the mean curvature operator of the Minkowski space. Its inverse, the mean curvature operator of the Euclidean space, also studied in \cite{Ber}, appears in Example \ref{examcoe}.}.
Then
$$F(x)=1-\sqrt{1-x^2},\ F_+^{-1}(x)=\sqrt{2x-x^2}.$$
In order for the conditions in Corollary \ref{thmexc} to be satisfied we need
$$(1+\l)F(c)<1,\quad (1+\l^{-1})F(c)<1,$$
that is
$$c<\min\left\{\frac{\sqrt{\l(\l+2)}}{\l+1},\frac{\sqrt{2\l+1}}{\l+1}\right\}.$$
In Figure \ref{figure2g} we plot how the period varies as a function of $c$ and $\l$. Observe how the period is decreasing in both parameters and $\lim_{c,\l\to 0}T(\l,c)=+\infty$.
 \begin{figure}[hhht]
  \center{\includegraphics[width=.5\textwidth]{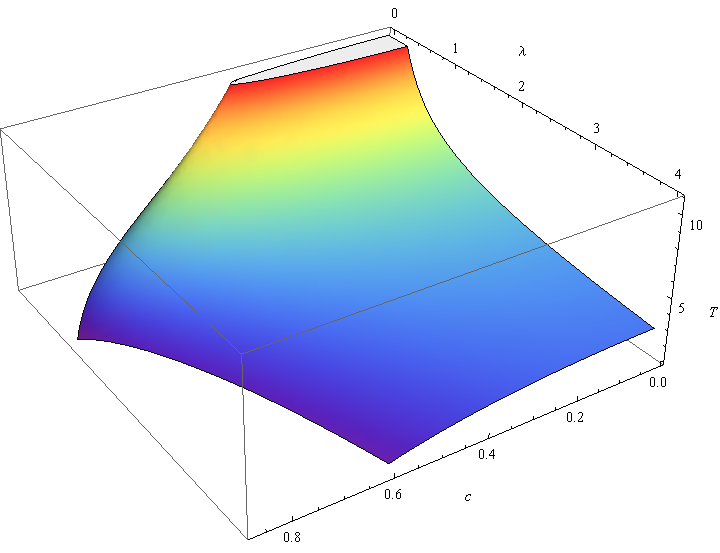}}\caption{Graph of the period $T$ function of $c$ and $\l$.}\label{figure2g}
  \end{figure}
\end{exa}
\begin{exa}\label{examcoe} Let $f:\bR\to(-1,1)$, $f(x):=x/\sqrt{1+x^2},\ x\in\bR$ and consider problem \eqref{eqpergpar}.  $f$ is effectively the inverse function of the one in the previous example.
Then
$$F(x)=\sqrt{1+x^2}-1,\ F_+^{-1}(x)=\sqrt{2x+x^2}.$$
The conditions in Corollary \ref{thmexc} are satisfied without any further restrictions.
In Figure \ref{figure2f} we plot how the period varies as a function of $c$ and $\l$. Observe in this plot how the period is decreasing in $\l$, increasing in $c$ and $\lim_{\l\to 0}T(c,\l)=\lim_{c\to +\infty}T(c,\l)=+\infty$.
\begin{figure}[h!]
 \center{\includegraphics[width=.5\textwidth]{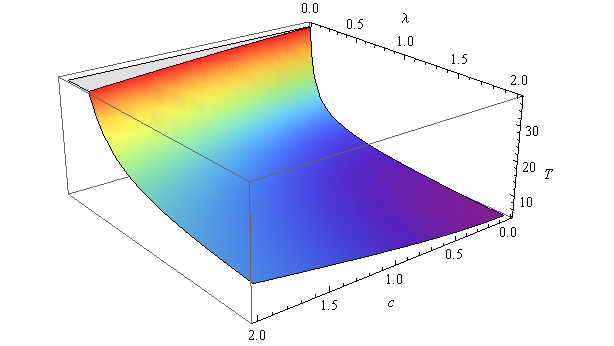}}\caption{Graph of the period $T$ function of $c$ and $\l$.}\label{figure2f}
\end{figure}
\end{exa}

\section{Problems with reflection}
Let us consider again the problem that motivated this paper in the Introduction, the obtaining of solutions of problem $\eqref{eqinvb}$ in the case $\phi(t)=-t$. Hence, consider all of the problems \eqref{eqinv1}--\eqref{eqperbis1} in the case $\phi(t)=-t$.\par
Observe that Lemma \ref{lem1b} (following Remark \ref{remlemex}) can be trivially extended to the following lemma.
\begin{lem}\label{lem1}
Let $f:(\tau_1,\tau_2)\to(\sigma_1,\sigma_2)$ an locally Lipchitz a.\,c. function with a.\,c. inverse. Then $x$ is a solution of the first order differential equation with involution
(\ref{eqinvb}) if and only if $x$ is a solution of the second order ordinary differential equation
(\ref{odeb}).\end{lem}

As was shown in Section 1, problem \eqref{odeb} is equivalent to problem \eqref{eqperbis1}.
We can now state the following corollary of Theorem \ref{thmex} regarding the periodicity of problem \eqref{eqinvb} as foreseen in Section 1.
\begin{cor}\label{thmper2} Let $f:(\tau_1,\tau_2)\to(\sigma_1,\sigma_2)$ an increasing locally Lipchitz a.\,c. function with a.\,c. inverse such that $0\in(\tau_1,\tau_2)$,  $f(0)=0$ and ${c}>0$. Assume $2F({c})<\min\{F(\tau_1),F(\tau_2)\}$. Then, if $x_{c}(t)$ is a solution of problem \eqref{eqper3} and we assume there exist $\ol {c}_1,\ol {c}_2\in\bR$, $\ol {c}_1<\ol {c}_2$, such that $2\max\{F(\ol c_1),F(\ol c_2)\}<\min\{F(\tau_1),F(\tau_2)\}$ and $(x_{\ol {c}_1}(b)-\ol {c}_1)(x_{\ol {c}_2}(b)-\ol {c}_2)<0$, then problem \eqref{eqinvb} must have at least a solution.
\end{cor}
We now give an example in which there is no need to find $\ol {c}_1,\ol {c}_2\in\bR$ in the conditions of Corollary \ref{thmper2} because the function determining the period has a simple inverse.
\begin{exa}
Take again $f(t):=|t|^{p-2}t$, $p>1,\ c>0$ and consider the problem
\begin{equation}\label{eqinvplap} x'(t)=|x(-t)|^{p-2}x(-t),\ t\in\bR,\ x(0)=c.\end{equation}
By Theorem \ref{thmper2} we have that the solutions of are periodic for every $c\ne 0$ and
$$T(c,1,p)=2^{\frac{2}{p}+1}\,c^{2-p}\frac{\Gamma\(\frac{1}{p}\)^2}{p\,\Gamma\(\frac{2}{p}\)}.$$
Consider now the problem 
\begin{equation}\label{eqinvplapper}x'(t)=|x(-t)|^{p-2}x(-t),\ t\in\bR,\ x(a)=x(b).\end{equation}
There is a unique solution for problem \eqref{eqinvplapper} for $p\in(0,2)\cup(2,+\infty)$. Just take the unique solution of problem \eqref{eqinvplap} with
$$c=\(\frac{b-a}{2^{\frac{2}{p}+1}}p\frac{\Gamma\(\frac{2}{p}\)}{\Gamma\(\frac{1}{p}\)^2}\)^\frac{1}{2-p}.$$
\end{exa}

\end{document}